\documentclass[11pt,a4paper]{amsart}
\pdfoutput=1
\usepackage[centering, margin=1in]{geometry}
\usepackage{amsmath,amsfonts,amsthm,amssymb,graphicx,mathrsfs,mathtools}
\usepackage[stretch=10]{microtype}
\usepackage[unicode,breaklinks=true]{hyperref}
\usepackage{xcolor}
\definecolor{dark-red}{rgb}{0.4,0.15,0.15}
\definecolor{dark-blue}{rgb}{0.15,0.15,0.4}
\definecolor{medium-blue}{rgb}{0,0,0.5}
\hypersetup{
	pdftitle={Test Vectors for Nonarchimedean Godement--Jacquet Zeta Integrals},
	pdfauthor={Peter Humphries},
	pdfnewwindow=true,
	colorlinks, linkcolor={dark-red},
	citecolor={dark-blue}, urlcolor={medium-blue}
}
\makeatletter
\newcommand*{\defeq}{\mathrel{\rlap{%
			\raisebox{0.3ex}{$\m@th\cdot$}}%
		\raisebox{-0.3ex}{$\m@th\cdot$}}%
	=}
\newcommand{\bigboxplus}{
	\mathop{
		\vphantom{\bigoplus} 
		\mathchoice
		{\vcenter{\hbox{\resizebox{\widthof{$\displaystyle\bigoplus$}}{!}{$\boxplus$}}}}
		{\vcenter{\hbox{\resizebox{\widthof{$\bigoplus$}}{!}{$\boxplus$}}}}
		{\vcenter{\hbox{\resizebox{\widthof{$\scriptstyle\oplus$}}{!}{$\boxplus$}}}}
		{\vcenter{\hbox{\resizebox{\widthof{$\scriptscriptstyle\oplus$}}{!}{$\boxplus$}}}}
	}\displaylimits 
}
\newcommand{\bigboxtimes}{
	\mathop{
		\vphantom{\bigotimes} 
		\mathchoice
		{\vcenter{\hbox{\resizebox{\widthof{$\displaystyle\bigotimes$}}{!}{$\boxtimes$}}}}
		{\vcenter{\hbox{\resizebox{\widthof{$\bigotimes$}}{!}{$\boxtimes$}}}}
		{\vcenter{\hbox{\resizebox{\widthof{$\scriptstyle\otimes$}}{!}{$\boxtimes$}}}}
		{\vcenter{\hbox{\resizebox{\widthof{$\scriptscriptstyle\otimes$}}{!}{$\boxtimes$}}}}
	}\displaylimits 
}
\makeatother
\allowdisplaybreaks
\newcommand{\Agp}{\mathrm{A}}
\renewcommand{\C}{\mathbb{C}}
\newcommand{\F}{\mathbb{F}}
\newcommand{\GL}{\mathrm{GL}}
\newcommand{\Mgp}{\mathrm{M}}
\newcommand{\Ngp}{\mathrm{N}}
\renewcommand{\O}{\mathrm{O}}
\newcommand{\OO}{\mathcal{O}}
\newcommand{\Pgp}{\mathrm{P}}
\newcommand{\pp}{\mathfrak{p}}
\newcommand{\R}{\mathbb{R}}
\newcommand{\Ss}{\mathscr{S}}
\newcommand{\WW}{\mathcal{W}}
\DeclareMathOperator{\antidiag}{antidiag}
\DeclareMathOperator{\diag}{diag}
\DeclareMathOperator{\Ind}{Ind}
\DeclareMathOperator{\Mat}{Mat}
\DeclareMathOperator*{\vol}{vol}
\numberwithin{equation}{section}
\newtheorem{theorem}[equation]{Theorem}
\newtheorem{lemma}[equation]{Lemma}
\newtheorem*{problem}{Test Vector Problem}
\theoremstyle{definition}
\newtheorem{definition}[equation]{Definition}
\begin{document}

\title[Test Vectors for Nonarchimedean Godement--Jacquet Zeta Integrals]{Test Vectors for Nonarchimedean Godement--Jacquet Zeta Integrals}

\author{Peter Humphries}

\address{Department of Mathematics, University College London, Gower Street, London WC1E 6BT, United Kingdom}

\email{\href{mailto:pclhumphries@gmail.com}{pclhumphries@gmail.com}}


\thanks{Research supported by the European Research Council grant agreement 670239.}

\begin{abstract}
Given an induced representation of Langlands type $(\pi,V_{\pi})$ of $\GL_n(F)$ with $F$ nonarchimedean, we show that there exist explicit choices of matrix coefficient $\beta$ and Schwartz--Bruhat function $\Phi$ for which the Godement--Jacquet zeta integral $Z(s,\beta,\Phi)$ attains the $L$-function $L(s,\pi)$.
\end{abstract}

\maketitle

\section{Introduction}

Let $F$ be a nonarchimedean local field with ring of integers $\OO$, maximal ideal $\pp$, and uniformiser $\varpi$, so that $\varpi \OO = \pp$ and $\OO/\pp \cong \F_q$ for some prime power $q$. We normalise the absolute value $|\cdot|$ on $F$ such that $|\varpi| = q^{-1}$.

Let $(\pi,V_{\pi})$ be a generic irreducible admissible smooth representation of $\GL_n(F)$, where $F$ is a nonarchimedean local field. Given a matrix coefficient $\beta(g) = \langle \pi(g) \cdot v_1, \widetilde{v_2}\rangle$ of $\pi$, where $v_1 \in V_{\pi}$ and $v_2 \in V_{\widetilde{\pi}}$, and given a Schwartz--Bruhat function $\Phi \in \Ss(\Mat_{n \times n}(F))$, we define the Godement--Jacquet zeta integral \cite{GJ72,Jac79}
\begin{equation}
\label{GJinteq}
Z(s,\beta,\Phi) \defeq \int_{\GL_n(F)} \beta(g) \Phi(g) \left|\det g\right|^{s + \frac{n - 1}{2}} \, dg,
\end{equation}
which is absolutely convergent for $\Re(s)$ sufficiently large. The test vector problem for Godement--Jacquet zeta integrals is the following.

\begin{problem}
Given a generic irreducible admissible smooth representation $(\pi,V_{\pi})$ of $\GL_n(F)$, determine the existence of $K$-finite vectors $v_1 \in V_{\pi}$, $\widetilde{v_2} \in V_{\widetilde{\pi}}$, and a Schwartz--Bruhat function $\Phi \in \Ss(\Mat_{n \times n}(F))$ such that
\[Z(s,\beta,\Phi) = L(s,\pi).\]
\end{problem}

The archimedean analogue of this problem has been resolved for $F = \C$ by Ishii \cite{Ish19} and for $F = \R$ by Lin \cite{Lin18}\footnote{The author has been unable to verify certain aspects of \cite{Lin18}. In particular, the functions constructed in \cite[(6.5) and (6.7)]{Lin18} are defined only on the maximal compact subgroup $K = \O(n)$ of $\GL_n(\R)$. For these functions to be elements of certain induced representations of $\GL_n(\R)$, they must transform under the action of diagonal matrices $a = \diag(a_1,\ldots,a_n) \in \Agp_n(\R)$ in a specified manner, and this action does not seem to be compatible with the definitions \cite[(6.5) and (6.7)]{Lin18} when $k \in K$ is taken to be a diagonal orthogonal matrix.}. For nonarchimedean $F$, the spherical case is resolved in \cite[Lemma 6.10]{GJ72}: one takes $v_1$ and $v_2$ to be spherical vectors and
\[\Phi(x) = \begin{dcases*}
1 & if $x \in \Mat_{n \times n}(\OO)$,	\\
0 & otherwise.
\end{dcases*}\]
We solve the ramified case of this problem.

\begin{theorem}
\label{mainthm}
Let $(\pi,V_{\pi})$ be a generic irreducible admissible smooth representation of $\GL_n(F)$ of conductor exponent $c(\pi) > 0$. Let $\beta(g)$ denote the matrix coefficient $\langle \pi(g) \cdot v^{\circ}, \widetilde{v^{\circ}}\rangle$, where $v^{\circ} \in V_{\pi}$ is the newform of $\pi$ normalised such that $\beta(1_n) = 1$. Define the Schwartz--Bruhat function $\Phi \in \Ss(\Mat_{n \times n}(F))$ by
\begin{equation}
\label{Phidefeq}
\Phi(x) \defeq \begin{dcases*}
\frac{\omega_{\pi}^{-1}\left(x_{n,n}\right)}{\vol(K_0(\pp^{c(\pi)}))} & if $x \in \Mat_{n \times n}(\OO)$ with $x_{n,1}, \ldots, x_{n,n - 1} \in \pp^{c(\pi)}$ and $x_{n,n} \in \OO^{\times}$,	\\
0 & otherwise,
\end{dcases*}
\end{equation}
where $\omega_{\pi}$ denotes the central character of $\pi$ and the congruence subgroup $K_0(\pp^{c(\pi)})$ is as in \eqref{K0eq}. Then for $\Re(s)$ sufficiently large,
\[Z(s,\beta,\Phi) = L(s,\pi).\]
\end{theorem}

\section{Induced Representations of Langlands Type}

Rather than working with generic irreducible admissible smooth representations, we will work in the more general setting of induced representations of Langlands type; see \cite[Section 1.5]{CP-S17} for further details.

Given representations $\pi_1, \ldots, \pi_r$ of $\GL_{n_1}(F), \ldots, \GL_{n_r}(F)$, where $n_1 + \cdots + n_r = n$, we form the representation $\pi_1 \boxtimes \cdots \boxtimes \pi_r$ of $\Mgp_{\Pgp}(F)$, where $\boxtimes$ denotes the outer tensor product and $\Mgp_{\Pgp}(F)$ denote the block-diagonal Levi subgroup of the standard parabolic subgroup $\Pgp(F) = \Pgp_{(n_1,\ldots,n_r)}(F)$ of $\GL_n(F)$. We then extend this representation trivially to a representation of $\Pgp(F)$. By normalised parabolic induction, we obtain an induced representation $\pi$ of $\GL_n(F)$,
\[\pi = \bigboxplus_{j = 1}^{r} \pi_j \defeq \Ind_{\Pgp(F)}^{\GL_n(F)} \bigboxtimes_{j = 1}^{r} \pi_j.\]

When $\pi_1,\ldots,\pi_r$ are irreducible and essentially square-integrable, $\pi_1 \boxplus \cdots \boxplus \pi_r$ is said to be an induced representation of Whittaker type; such a representation is admissible and smooth. Moreover, if each $\pi_j$ is of the form $\sigma_j \left|\det\right|^{t_j}$, where $\sigma_j$ is irreducible, unitary, and square-integrable, and $\Re(t_1) \geq \cdots \geq \Re(t_r)$, then $\pi$ is said to be an induced representation of Langlands type. Every irreducible admissible smooth representation $\pi$ of $\GL_n(F)$ is isomorphic to the unique irreducible quotient of some induced representation of Langlands type. If $\pi$ is also generic, then it is isomorphic to some (necessarily irreducible) induced representation of Langlands type.

An induced representation of Langlands type $(\pi,V_{\pi})$ is isomorphic to its Whittaker model $\WW(\pi,\psi)$, the image of $V_{\pi}$ under the map $v \mapsto \Lambda(\pi(\cdot) \cdot v)$, where $\Lambda : V_{\pi} \to \C$ is the unique (up to scalar multiplication) nontrivial Whittaker functional associated to an additive character $\psi$ of $F$. This is a continuous linear functional that satisfies
\[\Lambda\left(\pi(u) \cdot v\right) = \psi_n(u) \Lambda(v)\]
for all $v \in V_{\pi}$ and $u \in \Ngp_n(F)$, where $\Ngp_n(F)$ denotes the unipotent radical of the standard minimal parabolic subgroup and $\psi_n(u) \defeq \psi(u_{1,2} + u_{2,3} + \cdots + u_{n - 1,n})$.

An induced representation of Langlands type $\pi$ is said to be spherical if it has a $K$-fixed vector, where $K \defeq \GL_n(\OO)$. Such a spherical representation $\pi$ must be a principal series representation of the form $|\cdot|^{t_1} \boxplus \cdots \boxplus |\cdot|^{t_n}$; furthermore, the subspace of $K$-fixed vectors must be one-dimensional. This $K$-fixed vector, unique up to scalar multiplication, is called the spherical vector of $\pi$. In the induced model of $\pi$, the normalised spherical vector is the unique smooth right $K$-invariant function $f^{\circ} : \GL_n(F) \to \C$ satisfying
\[f^{\circ}(uag) = f^{\circ}(g) \delta_n^{1/2}(a) \prod_{i = 1}^{n} |a_i|^{t_i}\]
for all $u \in \Ngp_n(F)$, $a = \diag(a_1,\ldots,a_n) \in \Agp_n(F) \cong F^n$, the subgroup of diagonal matrices, and $g \in \GL_n(F)$, where $\delta_n(a) \defeq \prod_{i = 1}^{n} |a_i|^{n - 2i + 1}$ denotes the modulus character of the standard minimal parabolic subgroup, and normalised such that
\[f^{\circ}(1_n) = \prod_{i = 1}^{n - 1} \prod_{j = i + 1}^{n} \zeta_F(1 + t_i - t_j), \qquad \zeta_F(s) \defeq \frac{1}{1 - q^{-s}}.\]
The normalised spherical Whittaker function $W^{\circ}$ in the Whittaker model $\WW(\pi,\psi)$ is given by the analytic continuation of the Jacquet integral
\[W^{\circ}(g) \defeq \int_{\Ngp_n(F)} f^{\circ}(w_n u g) \overline{\psi_n}(u) \, du,\]
where $w_n = \antidiag(1,\ldots,1)$ is the long Weyl element. The Jacquet integral is absolutely convergent if $\Re(t_1) > \cdots > \Re(t_n)$ \cite[Section 3]{JS83} and extends holomorphically as a function of the complex variables $t_1,\ldots,t_n$ \cite{CS80}. The Haar measure on $\Ngp_n(F)$ is $du = \prod_{j = 1}^{n - 1} \prod_{\ell = j + 1}^{n} du_{j,\ell}$, where for $u_{j,\ell} \in F$, $du_{j,\ell}$ is the additive Haar measure on $F$ normalised to give $\OO$ volume $1$. With this normalisation of Haar measures and with $\psi$ an unramified additive character of $F$, the normalised spherical vector $W^{\circ} \in \WW(\pi,\psi)$ satisfies $W^{\circ}(1_n) = 1$.

\section{The Newform}

For each nonnegative integer $m$, we define the congruence subgroup $K_0(\pp^m)$ of $K$ by
\begin{equation}
\label{K0eq}
K_0\left(\pp^m\right) \defeq \left\{k \in K : k_{n,1},\ldots,k_{n,n - 1} \in \pp^m\right\}.
\end{equation}

\begin{theorem}[{\cite[Th\'{e}or\`{e}me (5)]{JP-SS81}}]
Let $(\pi,V_{\pi})$ be an induced representation of Langlands type of $\GL_n(F)$. Then either $\pi$ is spherical, so that
\[V_{\pi}^K \defeq \left\{v \in V_{\pi} : \pi(k) \cdot v = v \text{ for all } k \in K\right\}\]
is one-dimensional, or $\pi$ is ramified, in which case $V_{\pi}^K$ is empty and there exists a minimal positive integer $m = c(\pi)$ for which the vector subspace
\[V_{\pi}^{K_0(\pp^m)} \defeq \left\{v \in V_{\pi} : \pi(k) \cdot v = \omega_{\pi}(k_{n,n}) v \text{ for all } k \in K_0(\pp^m)\right\}\]
is nontrivial; moreover, $V_{\pi}^{K_0(\pp^{c(\pi)})}$ is one-dimensional.
\end{theorem}

\begin{definition}
The vector $v^{\circ} \in V_{\pi}^{K_0(\pp^{c(\pi)})}$, unique up to scalar multiplication, is called the newform of $\pi$. The nonnegative integer $c(\pi)$ is called the conductor exponent of $\pi$, where we set $c(\pi) = 0$ if $\pi$ is spherical.
\end{definition}

For each $m$, we may view $V_{\pi}^{K_0(\pp^m)}$ as the image of the projection map $\Pi^m : V_{\pi} \to V_{\pi}$ given by
\begin{align}
\label{Pimdefeq}
\Pi^m(v) & \defeq \int_K \xi^m(k) \pi(k) \cdot v \, dk,	\\
\label{ximdefeq}
\xi^m(k) & \defeq \begin{dcases*}
\frac{\omega_{\pi}^{-1}(k_{n,n})}{\vol(K_0(\pp^m))} & if $m > 0$ and $k \in K_0(\pp^m)$,	\\
1 & if $m = 0$ and $k \in K$,	\\
0 & otherwise.
\end{dcases*}
\end{align}
Here $dk$ is the Haar measure on the compact group $K$ normalised to give $K$ volume $1$. In particular, for any $v \in V_{\pi}$, we have that
\begin{equation}
\label{Piprojeq}
\Pi^{c(\pi)}(v) = \left\langle v, \widetilde{v^{\circ}}\right\rangle v^{\circ},
\end{equation}
where $v^{\circ} \in V_{\pi}^{K_0(\pp^{c(\pi)})}$ and $\widetilde{v^{\circ}} \in V_{\widetilde{\pi}}^{K_0(\pp^{c(\pi)})}$ are normalised such that $\langle v^{\circ}, \widetilde{v^{\circ}} \rangle = 1$.

We write $W^{\circ}$ for the newform in the Whittaker model $\WW(\pi,\psi)$ normalised such that $W^{\circ}(1_n) = 1$, where $\psi$ is an unramified additive character; we also normalise $v^{\circ} \in V_{\pi}$ and the Whittaker functional $\Lambda$ such that $\Lambda(v^{\circ}) = W^{\circ}(1_n) = 1$. Note that if $\pi$ is spherical, then the newform in the Whittaker model is precisely the normalised spherical Whittaker function.

A key property of $W^{\circ}$ is the fact that it is a test vector for certain Rankin--Selberg integrals.

\begin{theorem}[{Jacquet--Piatetski-Shapiro--Shalika \cite[Th\'{e}or\`{e}me (4)]{JP-SS81}, Jacquet \cite{Jac12}, Matringe \cite[Corollary 3.3]{Mat13}}]
\label{JPPSWhittakerthm}
Let $\pi$ be an induced representation of Langlands type, and let $W^{\circ} \in \WW(\pi,\psi)$ denote the newform in the Whittaker model. Then for any spherical representation of Langlands type $\pi'$ of $\GL_{n - 1}(F)$ with normalised spherical Whittaker function $W^{\prime\circ} \in \WW\left(\pi',\overline{\psi}\right)$, the $\GL_n \times \GL_{n - 1}$ Rankin--Selberg integral
\begin{equation}
\label{GLnxGLn-1eq}
\Psi\left(s,W^{\circ},W^{\prime\circ}\right) \defeq \int\limits_{\Ngp_{n - 1}(F) \backslash \GL_{n - 1}(F)} W^{\circ} \begin{pmatrix} g & 0 \\ 0 & 1 \end{pmatrix} W^{\prime\circ}(g) \left|\det g\right|^{s - \frac{1}{2}} \, dg
\end{equation}
is equal to the Rankin--Selberg $L$-function $L(s,\pi \times \pi')$.
\end{theorem}

Here the Haar measure on $\GL_n(F)$ is that induced from the Iwasawa decomposition $\GL_n(F) = \Ngp_n(F) \Agp_n(F) K$, namely $dg = du \, \delta_n^{-1}(a) \, d^{\times} a \, dk$, where $d^{\times}a = \prod_{i = 1}^{n} d^{\times} a_i$ with the multiplicative Haar measure on $F^{\times}$ given by $d^{\times} a_i = \zeta_F(1) |a_i|^{-1} \, da_i$.

\begin{theorem}[{Kim \cite[Theorem 2.1.1]{Kim10}}]
\label{Kimthm}
Let $\pi$ be an induced representation of Langlands type, and let $W^{\circ} \in \WW(\pi,\psi)$ denote the newform in the Whittaker model. Then for any spherical representation of Langlands type $\pi'$ of $\GL_n(F)$ with normalised spherical Whittaker function $W^{\prime\circ} \in \WW\left(\pi',\overline{\psi}\right)$, the $\GL_n \times \GL_n$ Rankin--Selberg integral
\begin{equation}
\label{GLnxGLneq}
\Psi\left(s,W^{\circ},W^{\prime\circ},\Phi^{\circ}\right) \defeq \int\limits_{\Ngp_n(F) \backslash \GL_n(F)} W^{\circ}(g) W^{\prime\circ}(g) \Phi(e_n g) \left|\det g\right|^s \, dg
\end{equation}
is equal to the Rankin--Selberg $L$-function $L(s,\pi \times \pi')$, where $e_n \defeq (0,\ldots,0,1) \in \Mat_{1 \times n}(F)$ and $\Phi^{\circ} \in \Ss(\Mat_{1 \times n}(F))$ is given by
\[\Phi^{\circ}(x_1,\ldots,x_n) \defeq \begin{dcases*}
\frac{\omega_{\pi}^{-1}(x_n)}{\vol(K_0(\pp^{c(\pi)}))} & if $c(\pi) > 0$, $x_1,\ldots,x_{n - 1} \in \pp^{c(\pi)}$, and $x_n \in \OO^{\times}$,	\\
1 & if $c(\pi) = 0$ and $x_1,\ldots,x_n \in \OO$,	\\
0 & otherwise.
\end{dcases*}\]
\end{theorem}

\section{A Propagation Formula}

We now present a propagation formula for spherical Whittaker functions. This is a recursive formula for a $\GL_n(F)$ Whittaker function in terms of a $\GL_{n - 1}(F)$ Whittaker function.

\begin{lemma}
\label{propagationformulalemma}
Let $\pi' = |\cdot|^{t_1'} \boxplus \cdots \boxplus |\cdot|^{t_n'}$ be a spherical representation of Langlands type of $\GL_n(F)$. Then the normalised spherical Whittaker function $W^{\prime\circ} \in \WW(\pi',\overline{\psi})$ satisfies
\begin{multline}
\label{WviacheckW}
W^{\prime\circ}(g) = \left|\det g\right|^{t_1' + \frac{n - 1}{2}} \int_{\GL_{n - 1}(F)} W_0^{\prime\circ}(h) \left|\det h\right|^{-t_1' - \frac{n}{2}}	\\
\times \int\limits_{\Mat_{(n - 1) \times 1}(F)} \Phi'\left(h^{-1} \begin{pmatrix} 1_{n - 1} & v \end{pmatrix} g\right) \psi(e_{n - 1} v) \, dv \, dh,
\end{multline}
where $W_0^{\prime\circ} \in \WW(\pi_0',\overline{\psi})$ is the normalised spherical Whittaker function of the spherical representation of Langlands type $\pi_0' \defeq |\cdot|^{t_2'} \boxplus \cdots \boxplus |\cdot|^{t_n'}$ of $\GL_{n - 1}(F)$ and $\Phi' \in \Ss(\Mat_{(n - 1) \times n}(F))$ is the Schwartz--Bruhat function
\[\Phi'(x) \defeq \begin{dcases*}
1 & if $x \in \Mat_{(n - 1) \times n}(\OO)$,	\\
0 & otherwise.
\end{dcases*}\]
\end{lemma}

\begin{proof}
Let $f^{\prime\circ}$ be the normalised spherical vector in the induced model of $\pi'$, so that
\begin{align}
\label{fnormeq}
f^{\prime\circ}(1_n) & = \prod_{i = 1}^{n - 1} \prod_{j = i + 1}^{n} \zeta_F(1 + t_i' - t_j'),	\\
\label{finducedeq}
f^{\prime\circ}(uag) & = f^{\prime\circ}(g) \delta_n^{1/2}(a) \prod_{i = 1}^{n} |a_i|^{t_i'},	\\
\label{fKeq}
f^{\prime\circ}(gk) & = f^{\prime\circ}(g)
\end{align}
for all $u \in \Ngp_n(F)$, $a = \diag(a_1,\ldots,a_n) \in \Agp_n(F)$, $g \in \GL_n(F)$, and $k \in K$. We claim that $f^{\prime\circ}$ is also given by the Godement section
\begin{equation}
\label{fviacheckf}
f^{\prime\circ}(g) \defeq \left|\det g\right|^{t_1' + \frac{n - 1}{2}} \int_{\GL_{n - 1}(F)} \Phi'\left(h^{-1} \begin{pmatrix} 0 & 1_{n - 1} \end{pmatrix} g\right) f_0^{\prime\circ}(h) \left|\det h\right|^{-t_1' - \frac{n}{2}} \, dh.
\end{equation}
Here $f_0^{\prime\circ}$ is the normalised spherical vector in the induced model of $\pi_0'$, so that
\begin{align}
\label{f0normeq}
f_0^{\prime\circ}(1_{n - 1}) & = \prod_{i = 2}^{n - 1} \prod_{j = i + 1}^{n} \zeta_F\left(1 + t_i' - t_j'\right),	\\
\label{f0inducedeq}
f_0^{\prime\circ}(u'a'h) & = f_0^{\prime\circ}(h) \delta_{n - 1}^{1/2}(a') \prod_{i = 2}^{n} |a_i'|^{t_i'},	\\
\label{f0Keq}
f_0^{\prime\circ}(hk') & = f_0^{\prime\circ}(h)
\end{align}
for all $u' \in \Ngp_{n - 1}(F)$, $a' = \diag(a_2',\ldots,a_n') \in \Agp_{n - 1}(F)$, $h \in \GL_{n - 1}(F)$, and $k' \in \GL_{n - 1}(\OO)$. We then insert the identity \eqref{fviacheckf} into the Jacquet integral
\[W^{\prime\circ}(g) \defeq \int_{\Ngp_n(F)} f^{\prime\circ}(w_n ug) \psi_n(u) \, du,\]
write $w_n = \left(\begin{smallmatrix} 0 & 1 \\ w_{n - 1} & 0 \end{smallmatrix}\right)$ and $u = \left(\begin{smallmatrix} u' & 0 \\ 0 & 1 \end{smallmatrix}\right) \left(\begin{smallmatrix} 1_{n - 1} & v \\ 0 & 1 \end{smallmatrix}\right)$ for $u' \in \Ngp_{n - 1}(F)$ and $v \in \Mat_{(n - 1) \times 1}(F)$, and make the change of variables $h \mapsto w_{n - 1} u' h$ to obtain the identity \eqref{WviacheckW}.

So it remains to show that $f^{\prime\circ}$ is indeed given by \eqref{fviacheckf}. We first show that this is an element of the induced model of $\pi'$, just as in \cite[Proposition 7.1]{Jac09}. We replace $g$ with $\left(\begin{smallmatrix} 1 & v \\ 0 & u' \end{smallmatrix}\right) \left(\begin{smallmatrix} a_1 & 0 \\ 0 & a' \end{smallmatrix}\right) g$, where $v \in \Mat_{1 \times (n - 1)}(F)$, $u' \in \Ngp_{n - 1}(F)$, $a_1 \in F^{\times}$, and $a' \in \Agp_{n - 1}(F)$. Upon making the change of variables $h \mapsto u' a' h$ and using \eqref{f0inducedeq}, we see that \eqref{finducedeq} is satisfied. Next, we check that $f^{\prime\circ}$ given by \eqref{fviacheckf} satisfies \eqref{fKeq}, which follows easily from the fact that $\Phi'(xk) = \Phi'(x)$ for all $x \in \Mat_{(n - 1) \times n}(F)$ and $k \in K$. Finally, we confirm the normalisation \eqref{fnormeq}. To see this, we use the Iwasawa decomposition $h = u' a' k'$ in \eqref{fviacheckf}, in which case the Haar measure is $dh = \delta_{n - 1}^{-1}(a') \, du' \, d^{\times} a' \, dk'$. The integral over $\GL_{n - 1}(\OO) \ni k'$ is trivial. We then make the change of variables $u' \mapsto {u'}^{-1}$, $a' \mapsto {a'}^{-1}$, so that
\[f^{\prime\circ}(1_n) = f_0^{\prime\circ}(1_{n - 1}) \int_{\Ngp_{n - 1}(F)} \int_{\Agp_{n - 1}(F)} \Phi'\begin{pmatrix} 0 & a' u' \end{pmatrix} \prod_{i = 2}^{n} |a_i'|^{-t_i'} \delta_{n - 1}^{1/2}(a') \left|\det a'\right|^{t_1' + \frac{n}{2}} \, d^{\times} a' \, du',\]
recalling \eqref{f0inducedeq}. Writing $du' = \prod_{j = 2}^{n - 1} \prod_{\ell = j + 1}^{n} du_{j,\ell}'$ and $d^{\times} a' = \prod_{i = 2}^{n} d^{\times} a_i'$ and making the change of variables $u_{j,\ell}' \mapsto {a_j'}^{-1} u_{j,\ell}'$, this becomes
\[f_0^{\prime\circ}(1_{n - 1}) \prod_{j = 2}^{n - 1} \prod_{\ell = j + 1}^{n} \int_{\OO} \, du_{j,\ell}' \prod_{i = 2}^{n} \int_{\OO \setminus \{0\}} |a_i'|^{1 + t_1' - t_i'} \, d^{\times} a_i'.\]
The integral over $\OO \ni u_{j,\ell}'$ is $1$, while the integral over $\OO \setminus \{0\} \ni a_i'$ is $\zeta_F(1 + t_1' - t_i')$. Recalling the normalisation \eqref{f0normeq} of $f_0^{\prime\circ}(1_{n - 1})$, we see that \eqref{fnormeq} is indeed satisfied.
\end{proof}

\section{Proof of \texorpdfstring{\hyperref[mainthm]{Theorem \ref*{mainthm}}}{Theorem 1.2}}

\begin{proof}[Proof of {\hyperref[mainthm]{Theorem \ref*{mainthm}}}]
Let $\pi$ be a ramified induced representation of Langlands type of $\GL_n(F)$, so that $c(\pi) > 0$, and let $\pi' = |\cdot|^{t_1'} \boxplus \cdots \boxplus |\cdot|^{t_n'}$ be an arbitrary spherical representation of Langlands type of $\GL_n(F)$. We insert the identity \eqref{WviacheckW} for the normalised spherical Whittaker function $W^{\prime\circ} \in \WW(\pi',\overline{\psi})$ into the $\GL_n \times \GL_n$ Rankin--Selberg integral \eqref{GLnxGLneq}. Just as in \cite[Equation (8.1)]{Jac09}, we fold the integration over $\Ngp_{n - 1}(F) \backslash \Ngp_n(F) \cong \Mat_{(n - 1) \times 1}(F) \ni v$ and make the change of variables $g \mapsto \left(\begin{smallmatrix} h & 0 \\ 0 & 1 \end{smallmatrix}\right) g$. In this way, we find that $\Psi(s,W^{\circ},W^{\prime\circ},\Phi^{\circ})$ is equal to
\begin{equation}
\label{Psifoldeq}
\int\limits_{\Ngp_{n - 1}(F) \backslash \GL_{n - 1}(F)} W_0^{\prime\circ}(h) \left|\det h\right|^{s - \frac{1}{2}} \int_{\GL_n(F)} W^{\circ}\left(\begin{pmatrix} h & 0 \\ 0 & 1 \end{pmatrix} g\right) \Phi(g) \left|\det g\right|^{s + t_1' + \frac{n - 1}{2}} \, dg \, dh,
\end{equation}
with $\Phi(x) \defeq \Phi^{\circ}(e_n x) \Phi'\left(\begin{pmatrix} 1_{n - 1} & 0 \end{pmatrix} x\right)$ as in \eqref{Phidefeq}.

We claim that
\begin{equation}
\label{PhiKeq}
\Phi(g) = \int_K \xi^{c(\pi)}(k) \Phi(k^{-1} g) \, dk,
\end{equation}
with $\xi^{c(\pi)}$ as in \eqref{ximdefeq}. Indeed, $\xi^{c(\pi)}(k)$ vanishes unless $k \in K_0(\pp^{c(\pi)})$, in which case $\Phi(k^{-1} g)$ vanishes unless $g \in \Mat_{n \times n}(\OO)$ with $g_{n,1}, \ldots, g_{n,n - 1} \in \pp^{c(\pi)}$ and $g_{n,n} \in \OO^{\times}$. Then as $k^{-1} \in K_0(\pp^{c(\pi)})$, it is easily checked that
\[\omega_{\pi}^{-1}(e_n k^{-1} g \prescript{t}{}{e_n}) = \omega_{\pi}(k_{n,n}) \omega_{\pi}^{-1}(g_{n,n}),\]
using the fact that $e_n k^{-1} g \prescript{t}{}{e_n} - e_n k^{-1} \prescript{t}{}{e_n} g_{n,n} \in \pp^{c(\pi)}$, $e_n k^{-1} \prescript{t}{}{e_n} k_{n,n} - 1 \in \pp^{c(\pi)}$, and $c(\omega_{\pi}) \leq c(\pi)$. Thus \eqref{PhiKeq} follows.

We insert \eqref{PhiKeq} into \eqref{Psifoldeq} and make the change of variables $g \mapsto kg$, so that the integral over $K \ni k$ is
\[\int_K W^{\circ}\left(\begin{pmatrix} h & 0 \\ 0 & 1 \end{pmatrix} k g\right) \xi^{c(\pi)}(k) \, dk = \Lambda\left(\pi \begin{pmatrix} h & 0 \\ 0 & 1 \end{pmatrix} \cdot \int_K \xi^{c(\pi)}(k) \pi(k) \cdot (\pi(g) \cdot v^{\circ}) \, dk\right).\]
We note that
\[\int_K \xi^{c(\pi)}(k) \pi(k) \cdot (\pi(g) \cdot v^{\circ}) \, dk = \Pi^{c(\pi)}\left(\pi(g) \cdot v^{\circ}\right) = \beta(g) v^{\circ},\]
where $\beta(g) \defeq \langle \pi(g) \cdot v^{\circ}, \widetilde{v^{\circ}}\rangle$, recalling \eqref{Pimdefeq} and \eqref{Piprojeq}, so that
\begin{equation}
\label{intKtobetaeq}
\int_K W^{\circ}\left(\begin{pmatrix} h & 0 \\ 0 & 1 \end{pmatrix} k g\right) \xi^{c(\pi)}(k) \, dk = \beta(g) W^{\circ}\begin{pmatrix} h & 0 \\ 0 & 1 \end{pmatrix}.
\end{equation}

Combining \eqref{Psifoldeq} with \eqref{PhiKeq} and \eqref{intKtobetaeq}, we find that
\[\Psi(s,W^{\circ},W^{\prime\circ},\Phi^{\circ}) = Z(s + t_1',\beta,\Phi) \Psi(s,W^{\circ},W_0^{\prime\circ}),\]
recalling the definitions \eqref{GJinteq} of the Godement--Jacquet zeta integral and \eqref{GLnxGLn-1eq} of the $\GL_n \times \GL_{n - 1}$ Rankin--Selberg integral. From \hyperref[Kimthm]{Theorems \ref*{Kimthm}} and \ref{JPPSWhittakerthm},
\[\Psi(s,W^{\circ},W^{\prime\circ},\Phi^{\circ}) = L(s,\pi \times \pi'), \qquad \Psi(s,W^{\circ},W_0^{\prime\circ}) = L(s,\pi \times \pi_0').\]
Moreover, \cite[(9.5) Theorem]{JP-SS83} implies that
\[L(s,\pi \times \pi') = L\left(s,\pi \times |\cdot|^{t_1'}\right) L\left(s,\pi \times \pi_0'\right) = L\left(s + t_1',\pi\right) L\left(s,\pi \times \pi_0'\right).\]
Since $L(s,\pi \times \pi_0')$ is not uniformly zero, we conclude that
\[Z(s + t_1',\beta,\Phi) = L(s + t_1',\pi).\qedhere\]
\end{proof}

\phantomsection
\addcontentsline{toc}{section}{Acknowledgements}

\subsection*{Acknowledgements}

The author would like to thank the anonymous referee for many helpful suggestions and corrections.

\end{document}